\documentclass[graybox]{svmult}

\usepackage{type1cm}        
\usepackage{makeidx}         
\usepackage{graphicx}        
\usepackage{multicol}        
\usepackage[bottom]{footmisc}

\usepackage{newtxtext}       %
\usepackage{newtxmath}       

\makeindex             

\begin{document}

\title*{Van Trees inequality, group equivariance, and estimation of principal subspaces}
\author{Martin Wahl}
\institute{Martin Wahl \at Humboldt-Universit\"{a}t zu Berlin, Unter den Linden 6, 10099 Berlin \at \email{martin.wahl@math.hu-berlin.de}}
%
%
\maketitle


\abstract{We establish non-asymptotic lower bounds for the estimation of principal subspaces. As applications, we obtain new results for the excess risk of principal component analysis and the matrix denoising problem. 
}

\section{Introduction}
Many learning algorithms and statistical procedures rely on the spectral decomposition of some empirical matrix or operator. Leading examples are principal component analysis (PCA) and its extensions to kernel PCA or manifold learning. In modern statistics and data science, such methods are typically studied in a high-dimensional or infinite-dimensional setting. Moreover, a major focus is on non-asymptotic results, that is, one seeks results that depend optimally on the underlying parameters, such as sample size and dimension (see, e.g., \cite{MR3611488,MR3980312} for two recent developments).

In this paper, we are concerned with non-asymptotic lower bounds for the estimation of principal subspaces, that is the eigenspace of the, say $d$, leading eigenvalues. As stated in \cite{MR3161458}, it is highly nontrivial to obtain such lower bounds which depend optimally on all underlying parameters, in particular the eigenvalues and $d$. In fact, in contrast to asymptotic settings, where one can apply the local asymptotic minimax theorem \cite{MR0400513}, it seems unavoidable to use some more sophisticated facts on the underlying parameter space of all orthonormal bases in order to obtain non-asymptotic lower bounds. A state-of-the-art result, obtained in \cite{MR3161458} and \cite{MR3161452}, provides a non-asymptotic lower bound for the spiked covariance model with two groups of eigenvalues. To state their result, consider the statistical model defined by   
\begin{equation}\label{eq_stat_experiment}
(\mathbb{P}_U)_{U\in O(p)},\qquad\mathbb{P}_U=\mathcal{N}(0, U\Lambda U^T)^{\otimes n},
\end{equation} 
where $O(p)$ denotes the orthogonal group, $\Lambda=\operatorname{diag}(\lambda_1,\dots,\lambda_p)$ is a diagonal matrix with $\lambda_1\geq \dots\geq \lambda_p> 0$, $\mathcal{N}(0, U\Lambda U^T)$ denotes a Gaussian distribution with expectation zero and covariance matrix $U\Lambda U^T$ and $n\geq 1$ is a natural number. This statistical model corresponds to observing $n$ independent $\mathcal{N}(0, U\Lambda U^T)$-distributed random variables $X_1,\dots,X_n$, and we will write $\mathbb{E}_U$ to denote expectation with respect to $X_1,\dots,X_n$ having law $\mathbb{P}_U$. Moreover, in this model, the $d$-th principal subspace (resp.~its corresponding orthogonal projection) is given by $P_{\leq d}(U)=\sum_{i\leq d}u_iu_i^T$, where $u_1,\dots ,u_p$ are the columns of $U\in O(p)$. 

\begin{theorem}[\cite{MR3161458}] 
Consider the statistical model \eqref{eq_stat_experiment} with $\lambda_1=\dots=\lambda_d>\lambda_{d+1}=\dots=\lambda_p>0$. Then there is an absolute constant $c>0$ such that
\[
\inf_{\hat P}\sup_{U\in O(p)}\mathbb{E}_U\|\hat P-P_{\leq d}(U)\|_{2}^2
\geq c\cdot \min\Big(\frac{d(p-d)}{n}\frac{\lambda_d\lambda_{d+1}}{(\lambda_d-\lambda_{d+1})^2},d,p-d\Big),
\]
where the infimum is taken over all estimators $\hat P=\hat P(X_1,\dots, X_n)$ with values in the class of all orthogonal projections on $\mathbb{R}^p$ of rank $d$ and $\|\cdot\|_{2}$ denotes the Hilbert-Schmidt (or Frobenius) norm. 
\end{theorem}
The proof is based on applying lower bounds under metric entropy conditions \cite{MR1742500} combined with the metric entropy of the Grassmann manifold \cite{MR1665590}. This (Grassmann) approach has been applied to many other principal subspaces estimation problems and spiked structures (see, e.g., \cite{CLM20,MR3766946,MR3396982,MR4036040}). In principle, it can also be applied to settings with decaying eigenvalues by considering spiked submodels. Yet, since this leads to lower bounds of a specific multiplicative form, it seems difficult to recover the optimal weighted eigenvalue expressions appearing in the asymptotic limit \cite{MR650934} and in the non-asymptotic upper bounds from \cite{MR4052188,MR4102689}.

To overcome this difficulty, \cite{W20} proposed a new approach based on a van Trees-type inequality with reference measure being the Haar measure on the special orthogonal group $SO(p)$. The key ingredient is to explore the group equivariance of the statistical model \eqref{eq_stat_experiment}, allowing to take advantage of the Fisher geometry of the model more efficiently. For instance, a main consequence of the developed theory is the following non-asymptotic analogue of the local asymptotic minimax theorem.

\begin{theorem}[\cite{W20}] 
Consider the statistical model \eqref{eq_stat_experiment} with $\lambda_1\geq \dots\geq \lambda_p> 0$. Then there are absolute constants $c,C>0$ such that, for every $h\geq C$, we have 
\begin{align*}
\inf_{\hat P}\int_{SO(p)}\mathbb{E}_{U}\|\hat P-P_{\leq d}(U)\|_{2}^2\,\pi_h(U)dU\geq c\cdot\sum_{i\leq d}\sum_{j>d}\min\Big(\frac{1}{n}\frac{\lambda_i\lambda_j}{(\lambda_i-\lambda_j)^2},\frac{1}{h^2p}\Big),
\end{align*}
where the infimum is taken over all $\mathbb{R}^{p\times p}$-valued estimators $\hat P=\hat P(X_1,\dots, X_n)$, $dU$ denotes the Haar measure on $SO(p)$, and the prior density $\pi_h$ is given by
$$
\pi_h(U)=\frac{\exp(hp\operatorname{tr}U)}{\int_{SO(p)}\exp(hp\operatorname{tr}U)\,dU}
$$
with $\operatorname{tr} U$ denoting the trace of $U$.
\end{theorem}
Clearly, for the (from a non-asymptotic point of view) optimal choice $h=C$, Theorem~2 implies Theorem 1, as can be seen from inserting $2d(p-d)/p\geq \min(d,p-d)$. Moreover, as shown in \cite[Section 1.3]{W20}, Theorem 2 can be used to derive tight non-asymptotic minimax lower bounds under standard eigenvalue conditions from functional PCA or kernel PCA, including exponentially and polynomially decaying eigenvalues. 
%

%

The goal of this paper is to extend the theory of \cite{W20} in several directions. First, we provide a lower bound for the excess risk of PCA. This loss function can be written as a weighted squared loss and a variation of the approach in \cite{W20} allows to deal with it. To achieve this, we establish a slightly complementary (and less general) van Trees-type inequality tailored for principal subspace estimation problems and dealing solely with the uniform prior. Interestingly, such uniform prior densities lead to trivial results in van Trees inequality from \cite{W20} (as well as in previous classical van Trees approaches \cite{MR1354456,MR2724359}). Finally, we provide lower bounds that are characterized by doubly substochastic matrices whose entries are bounded by the inverses of the different Fisher information directions, confirming previous non-asymptotic upper bounds that hold for the principal subspaces of the empirical covariance operator \cite[Section 2.3]{MR4102689}. 

\section{A van Trees inequality for the estimation of principal subspaces}\label{sec:van:trees:inequality}
In this section, we state a general van Trees-type inequality tailored for principal subspace estimation problems. Applications to more concrete settings are presented in Section \ref{sec:applications}.
Let $(\mathcal{X},\mathcal{F},(\mathbb{P}_U)_{U\in O(p)})$ be a statistical model with parameter space being the orthogonal group $O(p)$. Let $(A,\langle \cdot,\cdot\rangle)$ be a real inner product space of dimension $m\in \mathbb{N}$ and let $\psi:O(p)\rightarrow A$ be a derived parameter. We suppose that $O(p)$ acts (from the left, measurable) on $\mathcal{X}$ and $A$ such that 

\begin{description}[Type 1]
\item[(A1)] {$(\mathbb{P}_U)_{U\in O(p)}$ is $O(p)$-equivariant (i.e., $\mathbb{P}_{VU}(VE)=\mathbb{P}_{U}(E)$ for all $U,V\in O(p)$ and all $E\in\mathcal{F}$) and $U\mapsto \mathbb{P}_U(E)$ is measurable for all $E\in\mathcal{F}$.}
\item[(A2)] {$\psi$ is $O(p)$-equivariant (i.e., $\psi(VU)=V\psi(U)$ for all $U,V\in O(p)$).}
\item[(A3)] {$\langle Ua,Ub\rangle=\langle a,b\rangle$ for all $a,b\in A$ and all $U\in O(p)$.}
\end{description}
Condition (A1) says that for a random variable $X$ with distribution $\mathbb{P}_U$, the random variable $VX$ has distribution $\mathbb{P}_{VU}$. For more background on statistical models under group action, the reader is deferred to \cite{MR1089423,MR2431769} and also to \cite[Section 2.3]{W20}. 

Next, we specify the allowed loss functions. Let $(v_1,\dots,v_m):O(p)\rightarrow A^m$ be such that for all $j=1,\dots,m$,

\begin{description}[Type 1]
\item[(A4)] {$v_1(U),\dots,v_m(U)$ is an orthonormal basis of $A$ for all $U\in O(p)$.}
\item[(A5)] {$v_j$ are $O(p)$-equivariant (i.e., $v_j(VU)=Vv_j(U)$ for all $U,V\in O(p)$).}
\end{description}
For $w\in \mathbb{R}_{>0}^m$ we now define the loss function
\begin{align*}
l_w:O(p)\times A\rightarrow \mathbb{R}_{\geq 0},\qquad l_w(U,a)=\sum_{k=1}^mw_k\langle v_k(U),a-\psi(U)\rangle^2.
\end{align*}
If $w_1=\dots=w_m=1$, then $l_w$ does not depend on $v_1,\dots,v_m$ and is equal to the squared norm in $A$
\begin{align}\label{eq:choice:w1:HS:norm}
l_{(1,\dots,1)}(U,a)=\|a-\psi(U)\|^2=\langle a-\psi(U),a-\psi(U)\rangle.
\end{align}
For general $w$, the loss function $l_w$ is itself invariant in the sense that 
\begin{align}\label{eq:invariant:loss:function}
 l_w(VU,Va)=l_w(U,a)\quad \text{for all } U,V\in O(p),a\in A,
\end{align}
as can be seen from (A2), (A3) and (A5). For an estimator $\hat\psi(X)$ based on an observation $X$ from the model, the $l_w$-risk is defined as $\mathbb{E}_Ul_w(U,\hat{\psi}(X))$, where $\mathbb{E}_U$ denotes expectation when $X$ has distribution $\mathbb{P}_U$. 

In order to formulate our abstract main result, we also need some differentiability conditions on $\psi$ and the $v_j$. We assume that $\psi$ and $v_j$ are differentiable at the identity matrix $I_p$ in the sense that for all $\xi\in \mathfrak{so}(p)$, all $a\in A$ and all $j=1,\dots,m$, we have
\begin{description}[Type 1]
\item[(A6)] {$\lim_{t\rightarrow 0}\limits\big\langle \frac{\psi(\exp(t\xi))-\psi(I_p)}{t},a\big\rangle =\langle d\psi(I_p)\xi,a\rangle$.}
\item[(A7)] {$\lim_{t\rightarrow 0}\limits\big\langle \frac{v_j(\exp(t\xi))-v_j(I_p)}{t},a\big\rangle =\langle dv_j(I_p)\xi,a\rangle$.}
\end{description}
Here, $d\psi(I_p)\xi$ and $dv_j(I_p)\xi$ denote the directional derivatives at $I_p$ defined on the Lie algebra $\mathfrak{so}(p)$ on $SO(p)$ (i.e., the tangent space of $O(p)$ at $I_p$). Since $A$ is finite-dimensional, conditions (A6) and (A7) can also formulated in the norm-sense $\lim_{t\rightarrow 0}\|t^{-1}(\psi(\exp(t\xi))-\psi(I_p))-d\psi(I_p)\xi\|=0$ for all $\xi\in \mathfrak{so}(p)$. For some background on the special orthogonal group $SO(p)$ and its Lie algebra $\mathfrak{so}(p)$, see, e.g., \cite[Section 2.1]{W20}.

Recall that the $\chi^2$-divergence between two probability measures $\mathbb{P}\ll \mathbb{Q}$ is defined as $\chi^2(\mathbb{P},\mathbb{Q})=\int(\frac{d\mathbb{P}}{d\mathbb{Q}})^2\,d\mathbb{Q}-1$.

\begin{proposition}\label{main:lower:bound} Assume (A1)--(A7). Let $\xi_1,\dots, \xi_m\in \mathfrak{so}(p)$ be such that $\mathbb{P}_{\exp(t\xi_j)}\ll \mathbb{P}_{I_p}$ for all $j=1,\dots,m$ and all $t$ small enough. Suppose that there are $a_1,\dots,a_m\in(0,\infty]$ such that for all $j=1,\dots,m$,
\begin{align}\label{eq:chi2:Fisher}
\lim_{t\rightarrow 0}\frac{1}{t^2}\chi^2(\mathbb{P}_{\exp(t\xi_j)},\mathbb{P}_{I_p})=a_j^{-1},
\end{align}
Then, for all estimators $\hat\psi=\hat\psi(X)$ with values in $A$, we have 
\begin{align*}
\int_{O(p)}\mathbb{E}_Ul_w(U,\hat{\psi}(X))\,dU\geq \frac{\Big(\sum_{j=1}^m\limits\langle v_j(I_p),d\psi(I_p)\xi_j\rangle\Big)^2}{\sum_{j=1}^m\limits w_j^{-1}a_j^{-1}+\sum_{k=1}^m\limits w_k^{-1}\Big(\sum_{j=1}^m\limits\langle v_k(I_p),dv_j(I_p)\xi_j\rangle\Big)^2}.
\end{align*} 
\end{proposition}
\begin{remark}
In applications, the $a_j^{-1}\in [0,\infty)$ will be the different Fisher information directions. We use the inverse notation because it will be more suitable to solve the final optimization problem in Section \ref{sec:proof:appli:optimization}.
\end{remark}
\begin{remark}
Let us briefly compare Proposition \ref{main:lower:bound} to \cite[Proposition 1 and Theorem~3]{W20}, where a more general van Trees inequality is presented. In fact, the bound \cite[Theorem 3]{W20} has a more classical form and involves a general prior, an average over the prior in the numerator and Fisher informations of the prior in the denominator. Yet, while these Fisher informations are zero for the uniform prior considered in Proposition \ref{main:lower:bound}, the averages in the numerator are zero as well. Hence, \cite[Theorem 3]{W20} is trivial for the uniform prior. The reason that we can deal with the uniform prior lies in the fact that in addition to the equivariance  of the statistical model, we also require equivariance of the derived parameter and invariance of the loss function.   
\end{remark}

\section{Proof of Proposition \ref{main:lower:bound}}
We provide a proof which manifests Proposition \ref{main:lower:bound} as a Cramér-Rao-type inequality for equivariant estimators.
\subsection{Reduction to a pointwise risk}\label{sec:3.1}

We use \cite[Lemma 4]{W20} in order to reduce the Bayes risk of Proposition \ref{main:lower:bound} to a pointwise risk minimized over the class of all equivariant estimators. For completeness we briefly repeat the (standard) argument. Let $\tilde \psi$ be an arbitrary estimator with values in $A$. Without loss of generality we may restrict ourselves to estimators with bounded norm $\sup_{x\in\mathcal{X}}\|\tilde \psi(x)\|<\infty$ (indeed, by (A2) and (A3) we know that $\sup_{U\in O(p)}\|\psi(U)\|=C<\infty$. Hence, for $x\in\mathcal{X}$ such that $\|\tilde \psi(x)\|> C_w=2C(w_{\max}/w_{\min})^{1/2}$ with $w_{\max}=\max_k w_k<\infty$ and $w_{\min}=\min_k w_k>0$, we have $l_w^{1/2}(U,0)\leq w_{\max}^{1/2} C$, while $l_w^{1/2}(U,\tilde\psi(x))>w_{\min}^{1/2} C_w-w_{\max}^{1/2} C=w_{\max}^{1/2} C$). Hence, we can construct 
\begin{align*}
\hat\psi(x)=\int_{O(p)}V^T\tilde\psi(Vx)\,dV,\qquad x\in\mathcal{X}.
\end{align*}
By \cite[Lemma 4]{W20} this defines an $O(p)$-equivariant estimator (that is, it holds that $\hat\psi(Ux)=U\hat \psi(x)$ for all $x\in\mathcal{X}$ and all $U\in O(p)$) satisfying
\begin{align*}
 \int_{O(p)}\mathbb{E}_Ul_w(U,\tilde{\psi}(X))\,dU\geq \int_{O(p)}\mathbb{E}_Ul_w(U,\hat{\psi}(X))\,dU,
\end{align*}
where we used (A1) and the facts that the loss function $l_w$ is convex in the second argument and satisfies \eqref{eq:invariant:loss:function}. Moreover, using that $\hat\psi$ is $O(p)$-equivariant, it follows again from \cite[Lemma 4]{W20} that the risk $\mathbb{E}_Ul_w(U,\hat{\psi}(X))$ is constant over $U\in O(p)$. Hence, we arrive at 
\begin{align*}
 \inf_{\tilde{\psi}}\int_{O(p)}\mathbb{E}_Ul_w(U,\tilde{\psi}(X))\,dU\geq \inf_{\hat{\psi}\text{ $O(p)$-equivariant}}\mathbb{E}_{I_p}l_w(I_p,\hat{\psi}(X)),
\end{align*}
and it suffices to lower bound the right-hand side.

\subsection{A pointwise Cramér-Rao inequality for equivariant estimators}

The classical Cramér-Rao inequality provides a lower bound for the (co-)variance of unbiased estimators. In this section, we show that in our context, a similar lower bound can be proved for the class of all equivariant estimators.

\begin{lemma}\label{main:lemma}
Assume (A1)--(A7). Let $\xi_1,\dots, \xi_m\in \mathfrak{so}(p)$ be such that $\mathbb{P}_{\exp(t\xi_j)}\ll \mathbb{P}_{I_p}$ for all $j=1,\dots,m$ and all $t$ small enough. Suppose that there are $a_1,\dots,a_m\in(0,\infty]$ such that $\lim_{t\rightarrow 0}\chi^2(\mathbb{P}_{\exp(t\xi_j)},\mathbb{P}_{I_p})/t^2=a_j^{-1}$ for all $j=1,\dots,m$. Then, for any $O(p)$-equivariant estimator $\hat\psi(X)$ with values in $A$, we have
\begin{align*}
\mathbb{E}_{I_p}l_w(I_p,\hat{\psi}(X))\geq \frac{\Big(\sum_{j=1}^m\limits\langle v_j(I_p),d\psi(I_p)\xi_j\rangle\Big)^2}{\sum_{j=1}^m\limits w_j^{-1}a_j^{-1}+\sum_{k=1}^m\limits w_k^{-1}\Big(\sum_{j=1}^m\limits\langle v_k(I_p),dv_j(I_p)\xi_j\rangle\Big)^2}.
\end{align*}
\end{lemma}

\begin{proof} 
As shown in Section \ref{sec:3.1}, we can restrict ourselves to estimators $\hat\psi$ that are bounded. For $U_j=\exp(t\xi_j)$, $j=1,\dots,m$, consider the expression
\begin{align}\label{eq:basic:expression}
\sum_{j=1}^m\mathbb{E}_{I_p}\langle v_j(I_p),\hat\psi(X)-\psi(I_p)\rangle-\sum_{j=1}^m\mathbb{E}_{I_p}\langle v_j(I_p),\hat\psi(X)-\psi(U_j^T)\rangle.
\end{align}
Clearly \eqref{eq:basic:expression} is equal to
\begin{align}\label{eq:basic:expression:version1}
\sum_{j=1}^m\langle v_j(I_p),\psi(U_j^T)-\psi(I_p)\rangle.
\end{align}
On the other hand, using (A1)--(A3), (A5) and the equivariance of $\hat\psi$, we have
\begin{align*}
&\mathbb{E}_{I_p}\langle v_j(I_p),\hat\psi(X)-\psi(U_j^T)\rangle\\
&=\mathbb{E}_{U_j}\langle v_j(I_p),\hat\psi(U_j^TX)-\psi(U_j^T)\rangle=\mathbb{E}_{U_j}\langle v_j(I_p),U_j^T\hat\psi(X)-U_j^T\psi(I_p)\rangle\\
&=\mathbb{E}_{U_j}\langle v_j(U_j),\hat\psi(X)-\psi(I_p)\rangle=\mathbb{E}_{I_p}\frac{d\mathbb{P}_{U_j}}{d\mathbb{P}_{I_p}}(X)\langle v_j(U_j),\hat\psi(X)-\psi(I_p)\rangle.
\end{align*}
Hence, \eqref{eq:basic:expression} is also equal to
\begin{align}\label{eq:basic:expression:version2}
\mathbb{E}_{I_p}\sum_{j=1}^m\langle v_j(I_p),\hat\psi(X)-\psi(I_p)\rangle-\mathbb{E}_{I_p}\sum_{j=1}^m\frac{d\mathbb{P}_{U_j}}{d\mathbb{P}_{I_p}}(X)\langle v_j(U_j),\hat\psi(X)-\psi(I_p)\rangle.
\end{align}
Using \eqref{eq:basic:expression}--\eqref{eq:basic:expression:version2}, Parseval's identity, (A4) and the Cauchy-Schwarz inequality (twice), we arrive at
\begin{align*}
&\Big(\sum_{j=1}^m\langle v_j(I_p),\psi(U_j^T)-\psi(I_p)\rangle\Big)^2\\
&=\Big(\mathbb{E}_{I_p}\sum_{j=1}^m\langle v_j(I_p),\hat\psi(X)-\psi(I_p)\rangle-\mathbb{E}_{I_p}\sum_{j=1}^m\frac{d\mathbb{P}_{U_j}}{d\mathbb{P}_{I_p}}(X)\langle v_j(U_j),\hat\psi(X)-\psi(I_p)\rangle\Big)^2\\
&=\Big(\mathbb{E}_{I_p}\sum_{k=1}^m\Big\{\sum_{j=1}^m\langle v_j(I_p),v_k(I_p)\rangle-\sum_{j=1}^m\frac{d\mathbb{P}_{U_j}}{d\mathbb{P}_{I_p}}(X)\langle v_j(U_j),v_k(I_p)\rangle\Big\}\langle v_k(I_p),\hat\psi(X)-\psi(I_p)\rangle\Big)^2\\
&\leq \Big(\mathbb{E}_{I_p}\sum_{k=1}^mw_k\langle v_k(I_p),\hat\psi(X)-\psi(I_p)\rangle^2\Big)\\
&\quad\times \Big(\mathbb{E}_{I_p}\sum_{k=1}^mw_k^{-1}\Big\{\sum_{j=1}^m\langle v_j(I_p),v_k(I_p)\rangle-\sum_{j=1}^m\frac{d\mathbb{P}_{U_j}}{d\mathbb{P}_{I_p}}(X)\langle v_j(U_j),v_k(I_p)\rangle\Big\}^2\Big).
\end{align*}
The first term on the right-hand side is equal to the $l_w$-risk of $\hat\psi$ at $I_p$. Moreover, the second term can be written as 
\begin{align*}
&\sum_{k=1}^mw_k^{-1}\Big\{\Big(\sum_{j=1}^m\langle v_j(U_j)-v_j(I_p),v_k(I_p)\rangle\Big)^2+\mathbb{E}_{I_p}\Big(\sum_{j=1}^m\Big(\frac{d\mathbb{P}_{U_j}}{d\mathbb{P}_{I_p}}(X)-1\Big)\langle v_j(U_j),v_k(I_p)\rangle\Big)^2\Big\}.
\end{align*}
In particular, we have proved that
\begin{align}\label{eq:lower:bound:difference}
\mathbb{E}_{I_p}l_w(I_p,\hat{\psi}(X))\geq \frac{D^2}{\sum_{k=1}^mw_k^{-1}(D_k^2+\mathbb{E}_{I_p}(B_k+C_k)^2)}
\end{align}
with 
\begin{align*}
D&=\sum_{j=1}^m\langle v_j(I_p),\psi(U_j^T)-\psi(I_p)\rangle,\\
D_k&=\sum_{j=1}^m\langle v_j(U_j)-v_j(I_p),v_k(I_p)\rangle,\\
B_k&=\sum_{j=1}^m\Big(\frac{d\mathbb{P}_{U_j}}{d\mathbb{P}_{I_p}}(X)-1\Big)\langle v_j(I_p),v_k(I_p)\rangle,\\
C_k&=\sum_{j=1}^m\Big(\frac{d\mathbb{P}_{U_j}}{d\mathbb{P}_{I_p}}(X)-1\Big)\langle v_j(U_j)-v_j(I_p),v_k(I_p)\rangle.
\end{align*}
We now invoke a limiting argument to deduce Lemma \ref{main:lemma} from \eqref{eq:lower:bound:difference}. For this, recall that $U_j=\exp(t\xi_j)$, $\xi_j\in \mathfrak{so}(p)$, multiply numerator and denominator by $1/t^2$ and let $t\rightarrow 0$. First, by (A6) and (A7), we have
\begin{align*}
&\frac{1}{t}D\rightarrow -\sum_{j=1}^m\langle v_j(I_p),d\psi(I_p)\xi_j\rangle,\\
&\frac{1}{t}D_k\rightarrow\sum_{j=1}^m\langle v_k(I_p), dv_j(I_p)\xi_j\rangle
\end{align*}
as $t\rightarrow 0$. Moreover, by assumption \eqref{eq:chi2:Fisher}, we have
\begin{align*}
\frac{1}{t^2}\sum_{k=1}^m w_k^{-1}\mathbb{E}_{I_p}B_k^2=\frac{1}{t^2}\sum_{j=1}^m w_j^{-1}\chi^2(\mathbb{P}_{U_j},\mathbb{P}_{I_p})\rightarrow \sum_{j=1}^mw_j^{-1}a_j^{-1}\quad \text{as } t\rightarrow 0.
\end{align*}
On the other hand, $C_k$ is asymptotically negligible, as can be seen from
\begin{align*}
\frac{1}{t^2}\mathbb{E}_{I_p}C_k^2\leq \frac{1}{t^2}\Big(\sum_{j=1}^m\chi^2(\mathbb{P}_{U_j},\mathbb{P}_{I_p})\Big)\Big(\sum_{j=1}^m\langle v_j(U_j)-v_j(I_p),v_k(I_p)\rangle^2\Big)\rightarrow 0
\end{align*}
as $t\rightarrow 0$. Here, we used \eqref{eq:chi2:Fisher} and (A7). Thus, 
\begin{align*}
\frac{1}{t^2}\Big|\sum_{k=1}^m& w_k^{-1}\mathbb{E}_{I_p}(B_k+C_k)^2-\sum_{k=1}^m w_k^{-1}\mathbb{E}_{I_p}B_k^2\Big|\\
\leq \frac{1}{t^2}\sum_{k=1}^m& w_k^{-1}(2(\mathbb{E}_{I_p}B_k^2)^{1/2}(\mathbb{E}_{I_p}C_k^2)^{1/2}+\mathbb{E}_{I_p}C_k^2)\rightarrow 0
\end{align*}
as $t\rightarrow 0$. The proof now follows from inserting these limits into \eqref{eq:lower:bound:difference}.
\end{proof}

\section{Applications}\label{sec:applications}
In this section, we specialize our lower bounds in the context of principal component analysis (PCA) and a low-rank denoising model. In doing so, we will focus on the derived parameter 
\begin{align*}
\psi(U)=P_{\leq d}(U)=\sum_{i\leq d}u_iu_i^T,\qquad U\in O(p),
\end{align*}
where $1\leq d\leq p$ and $u_1,\dots ,u_p$ are the columns of $U\in O(p)$. This will correspond to the estimation of the $d$-th principal subspace. We discuss several loss functions based on the Hilbert-Schmidt distance and the excess risk in the reconstruction error.
\subsection{PCA and the subspace distance}
In this section, we consider the statistical model given in \eqref{eq_stat_experiment}
\begin{equation*}
(\mathbb{P}_U)_{U\in O(p)},\qquad\mathbb{P}_U=\mathcal{N}(0, U\Lambda U^T)^{\otimes n},
\end{equation*} 
with $\Lambda=\operatorname{diag}(\lambda_1,\dots,\lambda_p)$ and $\lambda_1\geq \dots\geq \lambda_p> 0$. The following theorem proved in Section \ref{sec:proof:appli} applies Proposition \ref{main:lower:bound} to the above model, derived parameter $P_{\leq d}$, and loss function given by the Hilbert-Schmidt distance (cf.~Section \ref{sec:proof:appli:specialization} below).
\begin{theorem}\label{theorem1} Consider the statistical model \eqref{eq_stat_experiment}. Then, for each $\delta>0$, we have 
\begin{align*}
\inf_{\hat P}\int_{O(p)} \mathbb{E}_U\|\hat{P}-P_{\leq d}(U)\|_2^2\,dU\geq I_\delta
\end{align*}
with infimum taken over all $\mathbb{R}^{p\times p}$-valued estimators $\hat P=\hat P(X_1,\dots,X_n)$ and
\begin{align*}
 I_\delta = \frac{1}{1+2\delta}\max\Big\{\sum_{i\leq d}\limits\sum_{j>d}\limits x_{ij}\ :\ & 0\leq x_{ij}\leq \tfrac{2}{n}\tfrac{\lambda_i\lambda_j}{(\lambda_i-\lambda_j)^2}\quad \text{for all }i\leq d,j>d,\\
                          & \sum_{i\leq d}\limits x_{ij}\leq \delta\quad \text{for all }j>d, \\
                         & \sum_{j> d}\limits x_{ij}\leq \delta\quad \text{for all }i\leq d\Big\}.
\end{align*}
\end{theorem}

\begin{remark}
We write $i\leq d$ for $i\in \{1,\dots,d\}$ and $j>d$ for $j\in \{d+1,\dots,p\}$.
\end{remark}

\begin{remark}
A (non-square) matrix $(x_{ij})$ is called doubly substochastic (cf.~\cite[Section~2]{B97}) if
\begin{eqnarray*}
& x_{ij}\geq 0 & \text{for all}\quad i,j,\\
& \sum_{i}\limits x_{ij}\leq 1\quad & \text{for all}\quad j, \\
& \sum_{j}\limits x_{ij}\leq 1\quad & \text{for all}\quad i. 
\end{eqnarray*}
Hence, choosing $\delta=1$, Theorem \ref{theorem1} holds with
\begin{align*}
I_1= \frac{1}{3}\max\Big\{\sum_{i\leq d}\sum_{j>d} x_{ij}:\ &(x_{ij}) \text{ doubly substochastic with }\\
& x_{ij}\leq \tfrac{2}{n}\tfrac{\lambda_i\lambda_j}{(\lambda_i-\lambda_j)^2}\quad \text{for all }i\leq d,j>d\Big\}.
\end{align*}
\end{remark}

\begin{remark}
That doubly substochastic matrices play a role is no coincidence. Such a structure also appears in the upper bounds for the principal subspaces of the empirical covariance operator (see, e.g., \cite{MR4102689}). To explain this, let $X_1,\dots,X_n$ be independent Gaussian random variables with expectation zero and covariance matrix $\Sigma$ and let $\hat\Sigma=n^{-1}\sum_{i=1}^nX_iX_i^T$ be the empirical covariance operator. Moreover, let $\lambda_1\geq \dots\geq \lambda_p$ (resp.~$\hat\lambda_1\geq \dots\geq \hat\lambda_p$) be the eigenvalues of $\Sigma$ (resp.~$\hat\Sigma$) and let $u_1,\dots,u_p$ (resp.~$\hat u_1,\dots,\hat u_p$) be the corresponding eigenvectors of $\Sigma$ (resp.~$\hat\Sigma$). Then, for $P_{\leq d}=\sum_{i\leq d}u_iu_i^T$ and $\hat P_{\leq d}=\sum_{i\leq d}\hat u_i\hat u_i^T$, we have (cf.~\cite{MR4052188,MR4102689})
\begin{align*}
\|\hat P_{\leq d}-P_{\leq d}\|_2^2=2\sum_{i\leq d}\sum_{j>d}x_{ij}\quad\text{with}\quad x_{ij}=\langle u_i,\hat u_j\rangle^2.
\end{align*}
There are two completely different possibilities to bound this squared Hilbert-Schmidt distance. First, by Bessel's inequality, we always have the trivial bounds 
\begin{align*}
 \sum_{i\leq d}x_{ij}\leq 1\quad\text{and}\quad \sum_{j> d}x_{ij}\leq 1.
\end{align*}
On the other hand, using perturbative methods, one has the central limit theorem
\begin{align*}
nx_{ij}=n\langle u_i,\hat u_j\rangle^2\stackrel{d}{\rightarrow}\mathcal{N}\Big(0,\frac{\lambda_i\lambda_j}{(\lambda_i-\lambda_j)^2}\Big),
\end{align*}
(see, e.g., \cite{And}, and also \cite{JW18} for a non-asymptotic version of this result). Hence, the lower bound in Theorem \ref{theorem1} can be interpreted as the fact that we can not do essentially better than the best mixture of trivial and perturbative bounds.
\end{remark}

\begin{remark}
A simple and canonical choice of the $x_{ij}$ in Theorem \ref{theorem1} is given by
\begin{align*}
 x_{ij}=\min\Big(\frac{2}{n}\frac{\lambda_i\lambda_j}{(\lambda_i-\lambda_j)^2},\frac{1}{p}\Big),
\end{align*}
in which case we rediscover the lower bound \cite[Theorem 1]{W20}. Yet, let us point out that the result in Theorem 2 is stronger in the sense that it allows for priors that are highly concentrated around $I_p$ (e.g.~$h$ of size $\sqrt{n}$), while Theorem \ref{theorem1} provides a lower bound for the uniform prior. 
\end{remark}

\begin{remark}
In general it seems difficult to find a simple closed form expression for the lower bound in Theorem \ref{theorem1}. For instance, an exception is given by the case $d=1$, in which case we have
\begin{align*}
I_\delta=\frac{1}{1+2\delta}\min\Big(\frac{2}{n}\sum_{j>1}\frac{\lambda_1\lambda_j}{(\lambda_1-\lambda_j)^2}, \delta\Big).
\end{align*}
\end{remark}

\begin{remark}
Using decision-theoretic arguments, the result can be extended to random variables with values in a Hilbert space; see \cite[Section 1.4]{W20} for the details.
\end{remark}

\subsection{PCA and the excess risk}
Theorem \ref{theorem1} provides a lower bound for the squared Hilbert-Schmidt distance $\|\hat{P}-P_{\leq d}(U)\|_2^2$. If the estimator $\hat P$ is itself an orthogonal projection of rank $d$, then $\|\hat{P}-P_{\leq d}(U)\|_2^2$ is equal to $\sqrt{2}$ times the Euclidean norm of the sines of the canonical angles between the corresponding subspaces (see, e.g., \cite[Chapter VII.1]{B97}). This so-called $\sin\Theta$ distance is a well-studied distance in linear algebra, numerical analysis and statistics (see, e.g., \cite{B97,I00,YWS15}). In the context of statistical learning, another important loss function arises if one introduces PCA as an empirical risk minimization problem with respect to the reconstruction error.

For $1\leq d\leq p$, let $\mathcal{P}_d$ be the set of all orthogonal projections $P:\mathbb{R}^p\rightarrow \mathbb{R}^p$ of rank $d$. Consider the statistical model defined by \eqref{eq_stat_experiment}. Then the reconstruction error is defined by
\begin{align*}
 R_U(P)=\mathbb{E}_U\|X-PX\|^2,\qquad P\in\mathcal{P}_d, U\in O(p),
\end{align*}
where $X$ is a random variable with distribution $\mathcal{N}(0,U\Lambda U^T)$ (under $\mathbb{E}_U$), and it is easy to see that (cf.~\cite{MR4102689})
\begin{align*}
P_{\leq d}(U)\in \arg\min_{P\in\mathcal{P}_d}R_U(P).
\end{align*}
Hence, the performance of $\hat P\in\mathcal{P}_d$ can be measured by its excess risk defined by
\begin{align}\label{eq:excess:risk:def}
 \mathcal{E}_U(\hat P)=R_U(\hat P)-\min_{P\in\mathcal{P}_d}R_U(P)=R_U(\hat P)-R_U(P_{\leq d}(U)).
\end{align}
In Section \ref{proof:theorem}, we show that $\mathcal{E}_U(\hat P)$ can be written in the form $l_w$ for some suitable choices for $A$, $v$ and $w$, and Proposition \ref{main:lower:bound} yields the following theorem.
\begin{theorem}\label{theorem2} Consider the statistical model \eqref{eq_stat_experiment} with the excess risk loss function from \eqref{eq:excess:risk:def}. Assume that $\lambda_1>\lambda_{d+1}$ and $\lambda_d>\lambda_p$. Let $1\leq r\leq d$ be the largest natural number such that $\lambda_r>\lambda_{d+1}$ and $d\leq s< p$ be the smallest natural number such that $\lambda_d>\lambda_{s+1}$. Then, for any $\mu\in[\lambda_{d+1},\lambda_d]$, we have
\begin{align*}
\inf_{\hat P}\int_{O(p)} \mathbb{E}_U\mathcal{E}_U(\hat P)\,dU\geq J_\mu
\end{align*}
with infimum taken over all $\mathcal{P}_d$-valued estimators $\hat P=\hat P(X_1,\dots,X_n)$ and
\begin{align*}
 J_\mu = \frac{1}{3}\max\Big\{\sum_{i\leq r}\limits\sum_{j>s}\limits x_{ij}\ :\ \ & 0\leq x_{ij}\leq \tfrac{1}{n}\tfrac{\lambda_i\lambda_j}{\lambda_i-\lambda_j}\quad \text{for all }i\leq r,j>s,\\
                          & \sum_{i\leq r}\limits x_{ij}\leq \mu-\lambda_j\quad \text{for all }j>s, \\
                         & \sum_{j> s}\limits x_{ij}\leq \lambda_i-\mu\quad \text{for all }i\leq r\Big\}.
\end{align*}
\end{theorem}

\begin{remark}
The lower bound is similar to the mixture bounds established in \cite{MR4102689}. In particular, as in the case of the Hilbert-Schmidt distance, the term $n^{-1}\lambda_i\lambda_j/(\lambda_i-\lambda_j)$ corresponds to the size of certain weighted projector norms, while the other two constrains correspond to trivial bounds. Hence, our lower bound strenghtens the reciprocal dependence of the excess risk on spectral gaps (the excess risk might be small in both cases, small and large gaps).
\end{remark}
 An important special case is given when $\lambda_d>\lambda_{d+1}$ (meaning that $r=s=d$) and when the last two restrictions in $J_\mu$ are satisfied for $x_{ij}=n^{-1}\lambda_i\lambda_j/(\lambda_i-\lambda_j)$. In particular, letting $\mu=(\lambda_d+\lambda_{d+1})/2$, they are satisfied if and only if 
\begin{align*}
 \frac{\lambda_{d+1}}{\mu-\lambda_{d+1}}\sum_{i\leq d} \frac{\lambda_i}{\lambda_i-\lambda_{d+1}}\leq n\quad\text{and}\quad \frac{\lambda_{d}}{\lambda_{d}-\mu}\sum_{j>d} \frac{\lambda_j}{\lambda_d-\lambda_{j}}\leq n,
\end{align*}
as can be seen from a monotonicity argument. A simple modification leads to the following corollary.
\begin{corollary}\label{cor:rel:rank} We have 
\begin{align*}
\inf_{\hat P}\int_{O(p)} \mathbb{E}_U\mathcal{E}_U(\hat P)\,dU\geq \frac{1}{3n}\sum_{i\leq d}\sum_{j>d}\frac{\lambda_i\lambda_j}{\lambda_i-\lambda_j},
\end{align*}
provided that
\begin{align}\label{eq:rel:rank:condition}
\frac{\lambda_{d}}{\lambda_d-\lambda_{d+1}}\Big(\sum_{i\leq d} \frac{\lambda_i}{\lambda_i-\lambda_{d+1}}+\sum_{j>d} \frac{\lambda_j}{\lambda_d-\lambda_{j}}\Big)\leq \frac{n}{2}.
\end{align}
\end{corollary}
%
%
%
Condition \eqref{eq:rel:rank:condition} is the main condition of \cite{MR4102689} under which perturbation bounds for the empirical covariance operator are developed (cf.~\cite[Remark 3.15]{MR4102689}). The involved eigenvalue expressions in Corollary \ref{cor:rel:rank} can be easily evaluated if the $\lambda_j$ have exponential or polynomial decay (cf.~\cite{RW18}). 
\begin{example}
If for some $\alpha>0$, we have $\lambda_j=e^{-\alpha j}$, $j=1,\dots,p$, then there are constants $c_1,c_2>0$ depending only on $\alpha$ such that 
\begin{align*}
\inf_{\hat P}\int_{O(p)} \mathbb{E}_U\mathcal{E}_U(\hat P)\,dU \geq c_1\frac{de^{-\alpha d}}{n},\quad\text{provided that}\quad d\leq c_2 n.
\end{align*}
A matching upper bound can be found in \cite[Section 2.4]{MR4102689}.
\end{example}

\begin{example}
Moreover, if for some $\alpha>0$, we have $\lambda_j=j^{-\alpha-1}$, $j=1,\dots,p$, then there are constants $c_1,c_2>0$ depending only on $\alpha$ such that 
\begin{align*}
\inf_{\hat P}\int_{O(p)} \mathbb{E}_U\mathcal{E}_U(\hat P)\,dU \geq c_1\frac{d^{2-\alpha}}{n},\quad\text{provided that}\quad d^2\log d\leq c_2 n.
\end{align*}
A matching upper bound (yet under a more restrictive condition on $d$) can be found in \cite[Section 2.4]{MR4102689}. The condition on $d$ can be further relaxed by using \cite{MR4052188}. 
\end{example}

\subsection{Low-rank matrix denoising}
For a diagonal matrix  $\Lambda=\operatorname{diag}(\lambda_1,\dots,\lambda_p)$ with $\lambda_1\geq \dots\geq \lambda_p\geq 0$, consider the family of probability measures $(\mathbb{P}_U)_{U\in O(p)}$ with $\mathbb{P}_U$ being the distribution of
\begin{align*}
X=U\Lambda U^T+\sigma W,
\end{align*}
where $\sigma>0$ and $W=(W_{ij})_{1\leq i,j\leq p}$ is drawn from the GOE ensemble, that is a symmetric random matrix whose upper triangular entries are independent zero mean Gaussian random variables with $\mathbb{E}W_{ij}^2=1$ for $1\leq i<j\leq p$ and $\mathbb{E}W_{ii}^2=2$ for $i=1,\dots,p$. Alternatively, this model can be defined on $\mathcal{X}=\mathbb{R}^{p(p+1)/2}$ with 
\begin{equation}\label{eq_stat_experiment:low:rank}
(\mathbb{P}_U)_{U\in O(p)},\qquad\mathbb{P}_U=\mathcal{N}(\operatorname{vech}(U\Lambda U^T), \sigma^2\Sigma_W),
\end{equation} 
where symmetric matrices $a\in\mathbb{R}^{p\times p}$ are transformed
into vectors using $\operatorname{vech}(a)=(a_{11},a_{21},\dots,a_{p1},a_{22},a_{32}\dots,a_{p2},\dots,a_{pp})\in \mathbb{R}^{p(p+1)/2}$, and $\Sigma_W$ is the covariance matrix of $\operatorname{vech}(W)$. The following theorem is the analogue of Theorem \ref{theorem1}.

\begin{theorem}\label{theorem3} Consider the statistical model \eqref{eq_stat_experiment:low:rank}. Then for each $\delta>0$, we have 
\begin{align*}
\inf_{\hat P}\int_{O(p)} \mathbb{E}_U\|\hat{P}-P_{\leq d}(U)\|_2^2\,dU\geq I'_\delta
\end{align*}
with infimum taken over all $\mathbb{R}^{p\times p}$-valued estimators $\hat P=\hat P(X_1,\dots,X_n)$ and
\begin{align*}
 I'_\delta = \frac{1}{1+2\delta}\max\Big\{\sum_{i\leq d}\limits\sum_{j>d}\limits x_{ij}\ :\ & 0\leq x_{ij}\leq \tfrac{2\sigma^2}{(\lambda_i-\lambda_j)^2}\quad \text{for all }i\leq d,j>d,\\
                          & \sum_{i\leq d}\limits x_{ij}\leq \delta\quad \text{for all }j>d, \\
                         & \sum_{j> d}\limits x_{ij}\leq \delta\quad \text{for all }i\leq d\Big\}.
\end{align*}
\end{theorem}

\begin{example}
Suppose that $\operatorname{rank}(\Lambda)=d\leq p-d$. Then, setting 
\begin{align*}
x_{ij}=\min\Big(\frac{\sigma^2}{\lambda_i^2},\frac{1}{p-d}\Big),\quad\text{we get}\quad I_1'\geq \frac{1}{3}\sum_{i\leq d}\min\Big(\frac{\sigma^2(p-d)}{\lambda_i^2}, 1\Big).
\end{align*}
Ignoring the minimum with $1$, $I_1'$ gives the size of a first order perturbation expansion for $X$ (see \cite{MR4052188} for more details and a corresponding upper bound).
\end{example}

\section{Proofs for Section \ref{sec:applications}}\label{sec:proof:appli}

In this section we show how Theorems \ref{theorem1}--\ref{theorem3} can be obtained by an application of Proposition \ref{main:lower:bound}. 

\subsection{Specialization to principal subspaces}\label{sec:proof:appli:specialization}

We start with specializing Proposition \ref{main:lower:bound} in the case where
\begin{align*}
\psi(U)=P_{\leq d}(U)=\sum_{i\leq d}u_iu_i^T=U\sum_{i\leq d}e_ie_i^TU^T,\qquad U\in O(p),
\end{align*}
where $1\leq d\leq p$ is a natural number. Here $u_i=Ue_i$ is the $i$-th column of $U$ and $e_1,\dots,e_p$ denotes the standard basis in $\mathbb{R}^p$. We consider $A=\mathbb{R}^{p\times p}$ endowed with the trace inner product $\langle a,b\rangle_2 =\operatorname{tr}(a^Tb)$, $a,b\in\mathbb{R}^{p\times p}$ and choose
\begin{align*}
v:O(p)\rightarrow \mathbb{R}^{p\times p},\qquad v(U)=(u_ku_l^T)_{1\leq k,l\leq p}.
\end{align*}
Hence, for $w\in\mathbb{R}_{>0}^{p\times p}$, we consider the loss function defined by
\begin{align*}
 l_w(U,a)=\sum_{k=1}^p\sum_{l=1}^p w_{kl}\langle u_ku_l^T, a-P_{\leq d}(U)\rangle_2^2.
\end{align*}
In particular, if $w_{kl}=1$ for all $k,l$, then $l_w(U,a)=\|a-P_{\leq d}(U)\|_2^2=\langle a-P_{\leq d}(U),a-P_{\leq d}(U)\rangle_2$ is the squared Hilbert-Schmidt (or Frobenius) distance. Note that, in contrast to Section \ref{sec:van:trees:inequality}, we consider a double index in this section. We equip $A$ with the group action given by conjugation $U\cdot a=UaU^T$, $a\in\mathbb{R}^{p\times p}$, $U\in O(p)$. Using this definition it is easy to see that (A2), (A3), (A4) and (A5) are satisfied. Moreover, the following lemma computes the derivatives in (A6) and (A7) in this case.

\begin{lemma}\label{lem:derivative}
For $\xi\in\mathfrak{so}(p)$, we have
 \begin{description}
 \item[(i)] $dP_{\leq d}(I_p)\xi=\xi\sum_{i\leq d}e_ie_i^T-\sum_{i\leq d}e_ie_i^T\xi$,
 \item[(ii)] $dv_{ij}(I_p)\xi=\xi e_ie_j^T-e_ie_j^T\xi$.
 \end{description}
In particular, for $i\neq j$ and $L^{(ij)}=e_ie_j^T-e_je_i^T\in\mathfrak{so}(p)$, we have
  \begin{description}
 \item[(i)] $dP_{\leq d}(I_p)L^{(ij)}=-dP_{\leq d}(I_p)L^{(ji)}=-e_ie_j^T-e_je_i^T$ if $i\leq d$ and $j>d$,
 \item[(ii)] $dv_{ij}(I_p)L^{(ij)}=e_ie_i^T-e_je_j^T$.
 \end{description}
\end{lemma}

\begin{remark}
We have $dP_{\leq d}(I_p)L^{(kl)}=0$ if $k,l\leq d$ or $k,l>d$.
\end{remark}

\begin{proof}
By definition, for $U\in SO(p)$ and $\xi\in \mathfrak{so}(p)$, we have $dP_{\leq d}(I_p)\xi=f'(0)$ with $f:\mathbb{R}\rightarrow \mathbb{R}^{p\times p},t\mapsto \sum_{i\leq d}(\exp(t\xi)e_i)(\exp(t\xi)e_i)^T$. Hence, using $(d/dt)\exp(t\xi)=\xi\exp(t\xi)$, (i) follows. Claim (ii) can be shown analogously and (iii) and (iv) follow from inserting $\xi=L^{(ij)}$ into (i) and (ii), respectively.
\end{proof}

\begin{corollary}\label{cor:lower:bound} Consider the above setting with $\psi=P_{\leq d}$. Suppose that (A1) holds and that there is a bilinear form $\mathcal{I}:\mathfrak{so}(p)\times \mathfrak{so}(p)\rightarrow \mathbb{R}$ such that
\begin{align}\label{eq:Fisher:info}
\lim_{t\rightarrow 0}\frac{1}{t^2}\chi^2(\mathbb{P}_{\exp(t\xi)},\mathbb{P}_{I_p})=\mathcal{I}(\xi,\xi)\qquad \text{for all }\xi\in\mathfrak{so}(p).
\end{align}
Let $I\subseteq \{1,\dots,d\}$ and $J\subseteq \{d+1,\dots,p\}$. Then, for every estimator $\hat P=\hat P(X)$ with values in $A=\mathbb{R}^{p\times p}$ and every $z_{ij}\in\mathbb{R}$, $i\in I$, $j\in J$, we have
\begin{align}
&\int_{O(p)}\mathbb{E}_Ul_w(U,\hat{P}(X))\,dU\label{eq:cor:lower:bound}\\
&\geq \frac{\Big(\sum_{i\in I}\limits\sum_{j\in J}\limits z_{ij}\Big)^2}{\sum_{i\in I}\limits\sum_{j\in J}\limits((w_{ij}+w_{ji})a_{ij})^{-1}z_{ij}^2+\sum_{i\in I}\limits w_{ii}^{-1}\Big(\sum_{j\in J}\limits z_{ij}\Big)^2+\sum_{j\in J}\limits w_{jj}^{-1}\Big(\sum_{i\in I}\limits z_{ij}\Big)^2},\nonumber
\end{align} 
where $a_{ij}^{-1}=\mathcal{I}(L^{(ij)},L^{(ij)})$ and $L^{(ij)}=e_ie_j^T-e_je_i^T$, $i\in I,j\in J$.
\end{corollary}

\begin{remark}\label{remark:cor:lower:bound}
The assumption that $w$ has strictly positive entries can be relaxed. To see this, let $\hat P=\hat P(X)$ be an estimator with $\int_{O(p)}\mathbb{E}_U\|\hat{P}(X)\|_2^2\,dU<\infty$, $w\in\mathbb{R}_{\geq 0}^{p\times p}$, and $I\subseteq \{1,\dots,d\}$ and $J\subseteq \{d+1,\dots,p\}$ such that $w_{ij},w_{ji},w_{ii},w_{jj}>0$ for all $i\in I$, $j\in J$. Then both sides of \eqref{eq:cor:lower:bound} are continuous at $w$ (apply the dominated convergence theorem to the left-hand side). Hence, by a limiting argument, \eqref{eq:cor:lower:bound} also hold under the latter conditions.
\end{remark}

\begin{proof}
We choose $\xi^{(ij)}=y_{ij}L^{(ij)}$ and $\xi^{(ji)}=-y_{ji}L^{(ji)}=y_{ji}L^{(ij)}$ for $i\in I$ and $j\in J$ and we set $\xi^{(lk)}=0$ in all other cases. Then, by Lemma \ref{lem:derivative}, the sum in the numerator of Proposition \ref{main:lower:bound} is equal to 
\begin{align*}
 &\sum_{i\in I}\limits\sum_{j\in J}\limits\langle v_{ij}(I_p),dP_{\leq d}(I_p)\xi^{(ij)}\rangle_2+ \sum_{j\in J}\limits\sum_{i\in I}\limits\langle v_{ji}(I_p),dP_{\leq d}(I_p)\xi^{(ji)}\rangle_2\\
 &=\sum_{i\in I}\limits\sum_{j\in J}\limits\big( y_{ij}\langle e_ie_j^T,-e_ie_j^T-e_je_i^T\rangle_2+y_{ji}\langle e_je_i^T,-e_ie_j^T-e_je_i^T\rangle_2\big)\\&=-\sum_{i\in I}\limits\sum_{j\in J}\limits (y_{ij}+y_{ji}).
\end{align*}
On the other hand, for $1\leq k,l\leq p$, the term in the squared brackets in the denominator is equal to
\begin{align*}
 \sum_{i\in I}\limits\sum_{j\in J}\limits&\langle v_{kl}(I_p),dv_{ij}(I_p)\xi^{(ij)}\rangle_2+ \sum_{j\in J}\limits\sum_{i\in I}\limits\langle v_{kl}(I_p),dv_{ji}(I_p)\xi^{(ji)}\rangle_2\\
 &=\sum_{i\in I}\limits\sum_{j\in J}\limits (y_{ij}\langle e_ke_l^T,e_ie_i^T-e_je_j^T\rangle_2+ y_{ji}\langle e_ke_l^T,e_ie_i^T-e_je_j^T\rangle_2)
\end{align*}
and the latter is equal to
\begin{align*}
\begin{cases}
 \sum_{j\in J}\limits (y_{kj}+y_{jk}), &k=l\in I,\\
 -\sum_{i\in I}\limits (y_{ik}+y_{ki}), &k=l\in J,\\
 0,& \text{else}.
\end{cases}
\end{align*}
Hence the second term in the denominator is equal to 
\begin{align*}
\sum_{i\in I}\limits w_{ii}^{-1}\Big(\sum_{j\in J}\limits (y_{ij}+y_{ji})\Big)^2+\sum_{j\in J}\limits w_{jj}^{-1}\Big(\sum_{i\in I}\limits (y_{ij}+y_{ji})\Big)^2
\end{align*}
Finally, the Fisher information term is equal to 
\begin{align*}
 \sum_{i\in I}\limits\sum_{j\in J}\limits \big(w_{ij}^{-1}\mathcal{I}(\xi^{(ij)},\xi^{(ij)})+  w_{ji}^{-1}\mathcal{I}(\xi^{(ji)},\xi^{(ji)})\big)=\sum_{i\in I}\limits\sum_{j\in J}\limits(w_{ij}^{-1}a_{ij}^{-1}y_{ij}^2+w_{ji}^{-1}a_{ij}^{-1}y_{ji}^2).
\end{align*}
Plugging all these formulas into Proposition \ref{main:lower:bound}, we get that, for every $y_{ij},y_{ji}\in\mathbb{R}$, $i\in I$, $j\in J$, the left-hand side in \eqref{eq:cor:lower:bound} is lower bounded by
\begin{align*}
\frac{\Big(\sum_{i\in I}\limits\sum_{j\in J}\limits (y_{ij}+y_{ji})\Big)^2}{\sum_{i\in I}\limits\sum_{j\in J}\limits(w_{ij}^{-1}a_{ij}^{-1}y_{ij}^2+w_{ji}^{-1}a_{ij}^{-1}y_{ji}^2)+\sum_{i\in I}\limits w_{ii}^{-1}\Big(\sum_{j\in J}\limits (y_{ij}+y_{ji})\Big)^2+\sum_{j\in J}\limits w_{jj}^{-1}\Big(\sum_{i\in I}\limits (y_{ij}+y_{ji})\Big)^2}.
\end{align*}
For $a,b\in(0,\infty]$ and $z\in\mathbb{R}$, it is easy to see that minimizing $a^{-1}x^2+b^{-1}y^2$ subject to $x+y=z$ leads to the value $(a+b)^{-1}z^2$. Applying this with $a=w_{ij}a_{ij}$, $b=w_{ji}a_{ij}$ and $z=z_{ij}$, the claim follows.
\end{proof}

\subsection{A simple optimization problem}\label{sec:proof:appli:optimization}

We now consider the optimization problem 
\begin{align}\label{eq:max:problem}
\max_{\substack{z_{ij}\in\mathbb{R}\\ i\in I,j\in J}}\frac{\Big(\sum_{i\in I}\limits\sum_{j\in J}\limits z_{ij}\Big)^2}{\sum_{i\in I}\limits\sum_{j\in J}\limits b_{ij}^{-1}z_{ij}^2+\sum_{i\in I}\limits w_{ii}^{-1}\Big(\sum_{j\in J}\limits z_{ij}\Big)^2+\sum_{j\in J}\limits w_{jj}^{-1}\Big(\sum_{i\in I}\limits z_{ij}\Big)^2},
\end{align}
where $w_{ii}$ and $w_{jj}$ are positive real numbers and $b_{ij}=(w_{ij}+w_{ji})a_{ij}\in (0,\infty]$, $i\in I$, $j\in J$. If $I$ or $J$ is a singleton, then a solution to \eqref{eq:max:problem} can be given explicitly.

\begin{lemma}\label{lem:max:problem:1}
Suppose that $I=\{1\}$ and that $w_{11}=w_{jj}=1$ for all $j\in J$. Then a solution of \eqref{eq:max:problem} is given by $z_{1j}=(1+b_{1j}^{-1})^{-1}$ with $b_{1j}=2a_{1j}$, leading to the maximum
\begin{align}\label{eq:max:problem:1}
 \frac{\sum_{j\in J}\limits(1+b_{1j}^{-1})^{-1}}{1+\sum_{j\in J}\limits(1+b_{1j}^{-1})^{-1}}\geq \frac{1}{4}\min\Big(\sum_{j\in J}\min(b_{1j},1), 1\Big)= \frac{1}{4}\min\Big(\sum_{j\in J}b_{1j}, 1\Big).
\end{align}
\end{lemma}

\begin{proof}
Obviously, the values for $z_{1j}$ given in Lemma \ref{lem:max:problem:1} lead to the expression on the left-hand side of \eqref{eq:max:problem:1}. The value defined through \eqref{eq:max:problem} is also upper bounded by the left-hand side of \eqref{eq:max:problem:1}, as can be seen by inserting the (Cauchy-Schwarz) inequality $(\sum_{j\in J}(1+b_{1j}^{-1})^{-1})^{-1}(\sum_{j\in J}z_{1j})^2\leq \sum_{j\in J}(1+b_{1j}^{-1})z_{1j}^2$ into \eqref{eq:max:problem}. Finally, the inequality in \eqref{eq:max:problem:1} follows from inserting $x/(1+x)\geq (1/2)\min(x, 1)$, $x\geq 0$.
\end{proof}

In general, it seems more difficult to give an explicit formula for \eqref{eq:max:problem} using only $b_{ij}$ and $\wedge$. Yet, the following lower bound is sufficient for our purposes. In the special case of Lemma \ref{lem:max:problem:1}, it gives the second bound in \eqref{eq:max:problem:1}.

\begin{lemma}\label{lem:max:problem}
For each $\delta>0$, the value defined through \eqref{eq:max:problem} is lower bounded by
\begin{align}
\text{maximize}\quad & \frac{1}{1+2\delta}\sum_{i\in I}\sum_{j\in J}x_{ij}&\nonumber\\
\text{subject to}\quad  & 0\leq x_{ij}\leq b_{ij}\qquad\ \ \   \text{for all}\quad i\in I,j\in J,\label{eq:constraints:doubly}\\
& \sum_{i\in I}\limits x_{ij}\leq \delta w_{jj}\qquad  \text{for all}\quad j\in J, \nonumber\\
& \sum_{j\in J}\limits x_{ij}\leq \delta w_{ii}\qquad\   \text{for all}\quad i\in I.\nonumber 
\end{align}
\end{lemma}

\begin{proof}
Let $z_{ij}=x_{ij}$ be real values satisfying the constraints in \eqref{eq:constraints:doubly}. Then we have
\begin{align*}
\sum_{i\in I}\limits\sum_{j\in J}\limits b_{ij}^{-1}z_{ij}^2+\sum_{i\in I}\limits w_{ii}^{-1}\Big(\sum_{j\in J}\limits z_{ij}\Big)^2+\sum_{j\in J}\limits w_{jj}^{-1}\Big(\sum_{i\in I}\limits z_{ij}\Big)^2\leq (1+2\delta)\sum_{i\in I}\sum_{j\in J}z_{ij}.
\end{align*}
Inserting this into \eqref{eq:max:problem}, the claim follows
\end{proof}

\begin{remark}
If $b_{ij}=\infty$, then the first constraint in \eqref{eq:constraints:doubly} can be written as $0\leq x_{ij}<\infty$.
\end{remark}

\subsection{End of proofs of the consequences}\label{proof:theorem}

\begin{proof}[Proof of  Theorem \ref{theorem1}]
By \cite[Lemma 1]{W20}, Condition \eqref{eq:Fisher:info} is satisfied with 
\begin{align*}
\mathcal{I}(\xi,\xi)=\frac{n}{2}\sum_{i,j=1}^p\xi_{ij}^2\frac{(\lambda_i-\lambda_j)^2}{\lambda_i\lambda_j},\qquad\xi\in\mathfrak{so}(p).
\end{align*}
Moreover, letting $O(p)$ act coordinate-wise on $\prod_{i=1}^n\mathbb{R}^p$, the statistical model in \eqref{eq_stat_experiment} satisfies (A1). Hence, applying Corollary~\ref{cor:lower:bound} with $I=\{1,\dots,d\}$ and $J=\{d+1,\dots,p\}$, $w_{kl}=1$ for all $k,l$ (leading to the Hilbert-Schmidt distance, cf.~\eqref{eq:choice:w1:HS:norm}), the claim follows from Lemma \ref{lem:max:problem}, using that $b_{ij}=2a_{ij}=2\mathcal{I}(L^{(ij)},L^{(ij)})^{-1}=2\lambda_i\lambda_j/(n(\lambda_i-\lambda_j)^2)\in (0,\infty]$.
\end{proof}

\begin{proof}[Proof of Theorem \ref{theorem2}]
The main remaining point is to show that the excess risk $\mathcal{E}_U(\hat P)$, $\hat P\in\mathcal{P}_d$, is of the form $l_w(U,\hat P)$ for some $w\in\mathbb{R}_{\geq 0}^{p\times p}$. This can be deduced from \cite[Lemma 2.6]{MR4102689}.

\begin{lemma}\label{lem:repr:excess:risk}
For $\hat P\in\mathcal{P}_d$ and $\mu\in[\lambda_{d+1},\lambda_d]$, we have
\begin{align*}
\mathcal{E}_U(\hat P)=\sum_{k=1}^p\sum_{l=1}^p w_{kl}\langle u_ku_l^T,\hat P-P_{\leq d}(U)\rangle^2_2=l_w(U,\hat P)
\end{align*}
with $w_{kl}=\lambda_k-\mu$ for $k\leq d$ and $w_{kl}=\mu-\lambda_k$ for $k>d$.
\end{lemma}

\begin{proof}
For brevity we write $P_{\leq d}(U)=P_{\leq d}$ and $P_k=P_k(U)=u_ku_k^T$. By \cite[Lemma 2.6]{MR4102689}, we have
\begin{align*}
\mathcal{E}_U(\hat P)=\sum_{k\leq d}(\lambda_k-\mu)\|P_k(I-\hat P)\|_2^2+\sum_{k>d}(\mu-\lambda_k)\|P_k\hat P\|_2^2.
\end{align*}
Inserting 
\begin{align*}
\|P_k(I-\hat P)\|_2^2&=\|P_k(P_{\leq d}-\hat P)\|_2^2,\qquad k\leq d,\\
\|P_k\hat P\|_2^2&=\|P_k(\hat P-P_{\leq d})\|_2^2=\|P_k(P_{\leq d}-\hat P)\|_2^2,\qquad k>d,
\end{align*}
we obtain 
\begin{align*}
\mathcal{E}_U(\hat P)&=\sum_{k\leq d}(\lambda_k-\mu)\|P_k(P_{\leq d}-\hat P)\|_2^2+\sum_{k>d}(\mu-\lambda_k)\|P_k(P_{\leq d}-\hat P)\|_2^2\\
&=\sum_{k\leq d}\sum_{l=1}^p(\lambda_k-\mu)\|P_k(P_{\leq d}-\hat P)P_l\|_2^2+\sum_{k>d}\sum_{l=1}^p(\mu-\lambda_k)\|P_k(P_{\leq d}-\hat P)P_l\|_2^2,
\end{align*}
and the claim follows from inserting the identity $\|P_k aP_l\|_2^2=\langle u_ku_l^T,a\rangle_2^2$, $a\in \mathbb{R}^{p\times p}$.
\end{proof}
Applying Lemma \ref{lem:repr:excess:risk}, we get
\begin{align*}
\inf_{\hat{P}\in\mathcal{P}_d}\int_{O(p)}\mathbb{E}_U\mathcal{E}_U(\hat P)\,dU= \inf_{\hat{P}\in\mathcal{P}_d}\int_{O(p)}\mathbb{E}_Ul_w(U,\hat P)\,dU,
\end{align*}
where the infimum is over all estimators $\hat{P}$ with values in $\mathcal{P}_d\subseteq \mathbb{R}^{p\times p}$. Hence, applying Corollary \ref{cor:lower:bound} and Remark \ref{remark:cor:lower:bound} with $w=(w_{kl})$ from Lemma \ref{lem:repr:excess:risk}, $I=\{1,\dots,r\}$, $J=\{s,\dots,p\}$ and 
\begin{align*}
b_{ij}=(w_{ij}+w_{ji})a_{ij}=(\lambda_i-\mu+(\mu-\lambda_j))\mathcal{I}(L^{(ij)},L^{(ij)})^{-1}=\frac{1}{n}\frac{\lambda_i\lambda_j}{\lambda_i-\lambda_j}\in(0,\infty),
\end{align*}
the claim follows from Lemma \ref{lem:max:problem}.
\end{proof}

\begin{proof}[Proof of Theorem \ref{theorem3}]
By \cite[Lemma 3]{W20}, Condition \eqref{eq:Fisher:info} is satisfied with 
\begin{align*}
\mathcal{I}(\xi,\xi)=\frac{1}{2\sigma^{2}}\sum_{i=1}^p\sum_{j=1}^p\xi_{ij}^2(\lambda_i-\lambda_j)^2,\qquad\xi\in\mathfrak{so}(p).
\end{align*}
Moreover, let $O(p)$ act on $\mathbb{R}^{p\times p}$ by conjugation. Since $VWV^T\stackrel{d}{=}W$, $V\in O(p)$, we have $VXV^T\stackrel{d}{=}VUX(VU)^T+\sigma W$, meaning that the statistical model in \eqref{eq_stat_experiment:low:rank} satisfies (A1). Hence, applying Corollary \ref{cor:lower:bound} with $w_{kl}=1$, $I=\{1,\dots,d\}$ and $J=\{d+1,\dots,p\}$, the claim follows from Lemma \ref{lem:max:problem}, using that $b_{ij}=2a_{ij}=2\mathcal{I}(L^{(ij)},L^{(ij)})^{-1}=2\sigma^2/(\lambda_i-\lambda_j)^2\in (0,\infty]$.
\end{proof}

\begin{acknowledgement}
This research has been supported by a Feodor Lynen Fellowship of the Alexander von Humboldt Foundation. The author would like to thank the two anonymous referees for their helpful comments and remarks.
\end{acknowledgement}

\bibliographystyle{plain}
\bibliography{lit}

\end{document}